\documentclass[11pt]{amsart}

\newtheorem{theorem}{Theorem}[section]
\newtheorem{lemma}[theorem]{Lemma}
\newtheorem{corollary}[theorem]{Corollary}
\newtheorem{proposition}[theorem]{Proposition}

\theoremstyle{definition}

\theoremstyle{remark}

\numberwithin{equation}{section}

\newcommand{\abs}[1]{\lvert#1\rvert}
\newcommand{\n}[1]{\Vert#1\Vert}

\newcommand{\C}{\mathbb C}
\newcommand{\D}{\mathbb D}
\newcommand{\R}{\mathbb R}
\newcommand{\U}{\mathbb U}

\begin{document}

\title{Cesaro Operators on the Hardy Spaces of the half-plane}

\author{Athanasios G. Arvanitidis }
\address{Department of Mathematics, University of Thessaloniki,
         54124 Thessaloniki, Greece}
\email{arvanit@math.auth.gr}
\thanks{The first author was supported with a graduate fellowship from Alexander S. Onassis Foundation.}

\author{Aristomenis G. Siskakis}
\address{Department of Mathematics, University of Thessaloniki,
         54124 Thessaloniki, Greece}
\email{siskakis@math.auth.gr}


\date{\today}


\begin{abstract}
In this article we  study  the Ces\`{a}ro
operator
$$
\mathcal{C}(f)(z)=\frac{1}{z}\int_{0}^{z}f(\zeta)\,d\zeta,
$$
and its  companion operator $\mathcal{T}$ on  Hardy spaces of the upper half plane.
We identify $\mathcal{C}$ and $\mathcal{T}$ as resolvents for  appropriate semigroups of composition operators and we find the norm and the spectrum in each case. The relation of
$\mathcal{C}$ and $\mathcal{T}$ with the corresponding Ces\`{a}ro operators on Lebesgue spaces $L^p(\R)$ of the boundary line  is also discussed.
\end{abstract}

\maketitle

\section{Introduction}
Let $\U=\{z\in \mathbb{C}:\textrm{Im}(z)>0\}$ denote the upper
half of the complex plane. For $0 < p<\infty,$ the Hardy space
$H^{p}(\U)$ is the space of analytic functions
$f:\U\rightarrow\mathbb{C}$ for which
$$
\n{f}_{H^{p}(\U)}=\sup_{y>0}\Bigl(\int_{-\infty}^{\infty}
\abs{f(x+iy)}^{p}dx\Bigr)^{\frac{1}{p}}<\infty.
$$
For $p=\infty$ we denote by $H^{\infty}(\U)$ the space of all
bounded analytic functions on $\U$ with the supremum norm.

The spaces $H^{p}(\U)$, $1\leq p\leq\infty$, are Banach spaces and
for $p=2$, $H^{2}(\U)$ is a Hilbert space. For the rest of the
paper we use the notation $\n{f}_{p}$ for $\n{f}_{H^{p}(\U)}$.

Let $1\leq p<\infty$ and suppose $f\in H^{p}(\U)$. Then $f$
satisfies the growth condition
\begin{equation}\label{growth estimate of H^p}
\abs{f(z)}^{p}\leq \frac{C\n{f}_{p}^{p}}{\textrm{Im}(z)}, \qquad
z\in \U,
\end{equation}
where $C$ is a constant.
Further the limit $\lim_{y\rightarrow0}f(x+iy)$ exists for almost
every $x$ in  $\mathbb{R}$ and we may define the boundary function
on $\mathbb{R}$, denoted by $f^*$, as
$$
f^*(x)=\lim_{y\rightarrow0}f(x+iy).
$$
This function is p-integrable and
\begin{equation}\label{norm of H^p}
\n{f}_{p}^{p}=\int_{-\infty}^{\infty}\abs{f^*(x)}^{p}dx.
\end{equation}
Thus $H^{p}(\U)$ can be viewed as a subspace of
$L^{p}(\mathbb{R})$. For more details on Hardy spaces
$H^p(\mathbb{U})$ see \cite{Du}, \cite{Ga}, \cite{Ho}.

For $f\in L^{p}(\mathbb{R})$ the well known
 Ces\`{a}ro transformation is defined by
\begin{equation}\label{S}
\mathtt{C}(f)(x)=\frac{1}{x}\int_0^x f(u)\,du
\end{equation}
for $x\in \R$, with an appropriate convention if $x=0$. This defines a bounded operator on $L^{p}(\mathbb{R})$ for $p>1$ \cite{BHS}. In particular if $f^{\ast} \in L^p(\R)$ is the boundary function of $f\in H^p(\U)$ the question arises whether the transformed function $\mathtt{C}(f^{\ast})$ is also the boundary function of an  $f\in H^p(\U)$.

In this article we consider the half-plane version of the Ces\`{a}ro operator, which is formally defined
$$
\mathcal{C}(f)(z)=\frac{1}{z}\int_{0}^{z}f(\zeta)d\zeta, \quad
f\in H^{p}(\U).
$$
It turns out that this formula defines an analytic function on $\U$ for each
$f\in H^{p}(\U)$ and the resulting operator is bounded on $H^p(\U),\, p>1$. In addition the companion operator
$$
\mathcal{T}(f)(z)=\int_{z}^{\infty}\frac{f(\zeta)}{\zeta}d\zeta,
\quad f\in H^{p}(\U),
$$
is also shown to be bounded on $H^p(\U)$ for $p\geq 1$. We find the norm and the spectrum of $\mathcal{C}, \mathcal{T}$ and we show that, for the boundary functions, we have  $\mathcal{C}(f)^{\ast}=\mathtt{C}(f^{\ast})$. The whole discussion is based on the observation that $\mathcal{C}$ and $\mathcal{T}$ can be obtained as resolvent operators for appropriate strongly continuous semigroups of simple composition operators on $H^p(\U)$.

\section{Related semigroups, strong continuity.}\label{The group of composition
operators}

For each $t\in\mathbb{R}$ consider the analytic self maps of $\mathbb{U}$
$$
\phi_{t}(z)=e^{-t}z,  \quad z\in\U,
$$
and for $1\leq p<\infty$, the corresponding weighted composition operators on $H^p(\U)$,
$$
T_{t}(f)(z)=e^{-\frac{t}{p}}f(\phi_t(z)), \quad f\in H^p(\U).
$$
For   $f\in H^p(\U)$ we have
\begin{align*}
\n{T_t(f)}_{p}^{p}&= e^{-t}\sup_{y>0}\Bigl(\int_{-\infty}^{+\infty}
\abs{f(e^{-t}x+ie^{-t}y)}^{p}dx\Bigr)\\
&=\sup_{v>0}\Bigl(\int_{-\infty}^{+\infty}\abs{f(u+iv)}^{p}du\Bigr)\\
&=\n{f}_{p}^{p},
\end{align*}
thus  each $T_{t}$ is an isometry on $H^{p}(\U)$. Further it is
easy to see that the family $\{T_{t}\}$ satisfies
$T_{t}T_{s}=T_{t+s}$ for each $t,s\in \R$ and $T_0=I$ the identity
operator, so $\{T_{t}\}$  is a group of isometries. We will use
this group or the positive and negative semigroups  $\{T_{t}, \
t\geq0\}$ and $\{T_{-t}, \ t\geq0\}$ in our study of the operators
$\mathcal{C}$ and $\mathcal{T}$.

\begin{proposition}\label{strong continuity}
For $1\leq p<\infty$  the group $\{T_{t}\}$ is strongly continuous on $H^{p}(\U)$.
The infinitesimal generator $\Gamma$ of $\{T_{t}\}$ is given by
$$
\Gamma(f)(z)=-zf'(z)-\frac{1}{p}f(z),
$$
and its domain is $D(\Gamma)= \{f\in H^{p}(\U): zf'(z)\in H^{p}(\U)\}.
$
\end{proposition}
\begin{proof}
For the strong continuity we need to show
\begin{equation}\label{strong continuity norm}
\lim_{t\to 0}\n{T_t(f)-f}_p=\lim_{t\rightarrow 0}\n{e^{-\frac{t}{p}}f\circ\phi_{t}-f}_{p}=0
\end{equation}
for every $f\in H^{p}(\U)$. Given such an $f$ consider the functions
$$
f_{s}(z)=f(z+is), \quad s>0.
$$
It is easy to see that $f_{s}\in
H^{p}(\U)$ for every $s$ and we have
\begin{align*}
\n{T_t(f)-f}_{p}&\leq
\n{T_t(f)-T_t(f_s)}_p+\n{T_t(f_s)-f_s}_p+\n{f_s-f}_{p}\\
&\leq
(\n{T_{t}}+1)\n{f_s-f}_{p}+
\n{T_t(f_s)-f_s}_{p}\\
&= 2\n{f_s-f}_{p}+
\n{e^{-\frac{t}{p}}f_s\circ\phi_t-f_s}_{p}.
\end{align*}
Now it is clear that the boundary function of   $f_s(z)$ is $f(x+is)$ and the boundary function of $e^{-\frac{t}{p}}f_s(e^{-t}z)$ is $e^{-\frac{t}{p}}f(e^{-t}x+is)$ for all $x \in \R$. Since
$$
\n{f_s-f}^p_{p}=\int_{-\infty}^{\infty}|f(x+is)-f^*(x)|^p\,dx\to
0,
$$
as $s\to 0$ (see for example \cite[Theorem 11.4]{Du}), it suffices
to show
$$
\lim_{t\to 0}\n{e^{-\frac{t}{p}}f_s\circ\phi_t-f_s}_{p}=0
$$
for each $s$. For a fixed such  $s$ and $a>0$ (to be specified later) we have
\begin{align*}
\n{e^{-\frac{t}{p}}f_s\circ\phi_{t}-f_s}_{p}^p &=
\int_{-\infty}^{\infty}|e^{-\frac{t}{p}}f(e^{-t}x+is)-f(x+is)|^p\,dx\\
&=\int_{-\infty}^{-2a}|e^{-\frac{t}{p}}f(e^{-t}x+is)-f(x+is)|^p\,dx \\
&\,\,\,\,\,+\int_{-2a}^{2a}|e^{-\frac{t}{p}}f(e^{-t}x+is)-f(x+is)|^p\,dx\\
&\,\,\,\,\,+\int_{2a}^{\infty}|e^{-\frac{t}{p}}f(e^{-t}x+is)-f(x+is)|^p\,dx\\
&=I_1(t)+I_2(t)+I_3(t).
\end{align*}
To estimate the integral $I_1(t)$ we use the standard inequality
$$
|e^{-\frac{t}{p}}f(e^{-t}x+is)-f(x+is)|^p\leq 2^p(e^{-t}|f(e^{-t}x+is)|^p+|f(x+is)|^p)
$$
and the  change  of variable $u=e^{-t}x$ to obtain
\begin{align*}
I_1(t) &\leq 2^p\int_{-\infty}^{-2ae^{-t}}|f(u+is)|^p\,du+2^p\int_{-\infty}^{-2a}|f(x+is)|^p\,dx\\
&\leq 2^{p+1}\int_{-\infty}^{-a}|f(x+is)|^p\,dx,
\end{align*}
valid for $0<t<\log2$. An analogous estimate for $I_3(t)$ gives
$$
I_3(t)\leq 2^{p+1}\int_{a}^{\infty}|f(x+is)|^p\,dx,
$$
for $0<t<\log2$.

Let $\epsilon>0$ be given. Since $f_s \in H^p(\U)$ we can choose  $a>0$
such that for $0<t<\log2$,
$$
I_1(t)+I_3(t)\leq 2^{p+1}\int_{\R\setminus [-a,\, a]}|f(x+is)|^p\,dx  < \epsilon.
$$
We now deal with the second integral $I_2(t)$. Using the growth estimate  \eqref{growth
estimate of H^p} we have
\begin{align*}
|e^{-\frac{t}{p}}f(e^{-t}x+is)-f(x+is)|^p&\leq 2^p|e^{-\frac{t}{p}}f(e^{-t}x+is)|^p+2^p|f(x+is)|^p\\
&\leq 2^pC\frac{\n{e^{-\frac{t}{p}}f(e^{-t}x+is)}_p^p}{\mbox{Im}(e^{-t}x+is)}
+2^pC\frac{\n{f(x+is)}_p^p}{\mbox{Im}(x+is)}\\
&\leq 2^pC\frac{\n{f}_p^p}{ s}+2^pC\frac{\n{f}_p^p}{ s}\\
&=2^{p+1}C\frac{\n{f}_p^p}{s}.
\end{align*}
This says in particular that $|e^{-\frac{t}{p}}f(e^{-t}x+is)-f(x+is)|^p$ is a uniformly bounded
function of $x$ on the interval $[-2a, 2a]$ for each $t>0$. Since this function has pointwise limit $0$ as $t\to 0$,  an application of Lebesgue's
dominated convergence theorem gives
$$
\lim_{t\to 0}\int_{-2a}^{2a}|e^{-\frac{t}{p}}
f(e^{-t}x+is)-f(x+is)|^p\,dx=0.
$$
Putting together all previous estimates we conclude  $\lim_{t\to 0}\n{T_t(f)-f}_p=0$ and the strong continuity follows.

By definition the domain
$D(\Gamma)$ of $\Gamma$, consists of all $f\in H^p(\U)$ for which the limit
$
\lim_{t \to 0}\frac{e^{-\frac{t}{p}} f\circ\phi_t-f}{t}
$
exists in $H^p(\U)$ and
$$
\Gamma(f)=\lim_{t \to 0}\frac{e^{-\frac{t}{p}}f\circ\phi_t-f}{t}, \quad f\in D(\Gamma).
$$
The growth estimate (\ref{growth estimate of H^p}) shows that
convergence in the norm of $H^p(\U)$ implies in particular pointwise
convergence, therefore for each $f\in D(\Gamma)$,
\begin{align*}
\Gamma(f)(z)&=\lim_{t\to 0}\frac{e^{-\frac{t}{p}} f(e^{-t}z)-f(z)}{t}\\
&=\frac{\partial}{\partial t}\left(e^{-\frac{t}{p}}f(e^{-t}z)\right)\Bigr|_{t=0}\\
&=-zf'(z)-\frac{1}{p}f(z).
\end{align*}
This shows that $D(\Gamma)\subseteq \{f\in H^p(\U): zf'(z)\in H^p(\U)\}$.

Conversely let $f\in H^p(\U)$ such that $zf'(z)\in H^p(\U)$. Then
we can write for $z\in \U$,
\begin{align*}
T_t(f)(z)-f(z)&=\int_0^t\frac{\partial}{\partial s} \left(e^{-\frac{s}{p}}f(\phi_{s}(z))\right)\,ds\\
&=\int_{0}^{t}-e^{-\frac{s}{p}}\phi_{s}(z)f'(\phi_{s}(z))
-\frac{1}{p}e^{-\frac{s}{p}}f(\phi_s(z))\,ds\\
&=\int_{0}^{t}T_s(F)(z)\,ds,
\end{align*}
where  $F(z)=-zf'(z)-\frac{1}{p}f(z)$ is a function in
$H^{p}(\U)$. Thus
$$
\lim_{t\to 0}\frac{T_{t}(f)-f}{t}=\lim_{t\to
0}\frac{1}{t}\int_{0}^{t}T_{s}(F)\,ds.
$$
From the general theory of strongly continuous (semi)groups, for
$F\in H^p(\U)$  the latter limit exists in $H^p(\U)$ and is equal
to $F$. Thus
$$
\lim_{t\to 0}\frac{T_{t}(f)-f}{t}=F.
$$
This says that $D(\Gamma)\supseteq \{f\in H^p(\U): zf'(z)\in H^p(\U)\}$
completing the  proof.
\end{proof}

\vspace{0.4cm}\noindent
\textbf{Remark.} An  immediate corollary of the previous proposition
is that the two semigroups $\{T_{t}, \ t\geq0\}$ and $\{T_{-t}, \ t\geq0\}$
are strongly continuous on $H^{p}(\U)$ and their infinitesimal generators
are $\Gamma$ and $-\Gamma$ respectively.

\begin{lemma}\label{(z)^l}
Let $0<p<\infty$ and  $\lambda\in \mathbb{C}$, then \\
(i) $e_{\lambda}(z)=z^{\lambda}\notin H^{p}(\U)$.\\
(ii) $h_{\lambda}(z)=(i+z)^{\lambda}\in H^{p}(\U)$ if and only if
$\textrm{Re}(\lambda) <-\frac{1}{p}$.
\end{lemma}

\begin{proof} This can be proved by direct calculation of the norms.
We give an alternative
argument involving the  following well known characterization of membership
in the Hardy space of the half plane. For a function $f$ analytic on $\U$,
$$
f\in H^{p}(\U)\quad \mbox{if and only if}\quad \psi'(z)^{1/p}f(\psi(z))\in H^{p}(\mathbb{D}),
$$
where $\psi(z)=i\frac{1+z}{1-z}$, a conformal map from the unit
disc $\mathbb{D}=\{|z|<1\}$ onto $\U$, and
$H^{p}(\mathbb{D})$ is the usual Hardy space of the disc. We find
\begin{equation}{\label{h}}
\psi'(z)^{1/p}h_{\lambda}(\psi(z))=\frac{c_1}{(1-z)^{\frac{2}{p}+\lambda}},
\end{equation}
and
\begin{equation}{\label{e}}
\psi'(z)^{1/p}e_{\lambda}(\psi(z))=c_2\frac{(1+z)^{\lambda}}{(1-z)^{\frac{2}{p}+\lambda}},
\end{equation}
where $c_1, c_2$ are nonzero constants.

Next recall that if $\nu$ is a complex number,   $(1-z)^{\nu}\in H^p(\mathbb{D})$
if and only if $\textrm{Re}(\nu)>-1/p$ (this follows for example from \cite[page 13, Ex. 1]{Du}).
Applying this to (\ref{h}) we obtain $h_{\lambda}\in H^p(\U)$ if and only if
$\textrm{Re}(-\frac{2}{p}-\lambda) >-1/p$  and this gives the desired conclusion.

In the case of
$e_{\lambda}$, the right hand side of (\ref{e}) belongs to  $H^p(\mathbb{D})$ if and only if
both terms $(1+z)^{\lambda}$ and $(1-z)^{-\frac{2}{p}-\lambda}$ belong to $H^p(\mathbb{D})$,
because each of the two terms is analytic and nonzero at the point where the other term has a singularity.
Thus for $e_{\lambda}$ to belong to $H^p(\U)$, both conditions
$\textrm{Re}(\lambda)>-1/p$ and $\textrm{Re}(-\frac{2}{p}-\lambda)>-1/p$ must be satisfied,
which is impossible.
\end{proof}

For an operator $A$ denote  by $\sigma_{\pi}(A)$ the set of eigenvalues of $A$, by $\sigma(A)$ the spectrum of $A$
and by $\rho(A)$ the resolvent set of $A$ on $H^p(\mathbb{U})$.

\begin{proposition}\label{spectr(Ã)} Let $1\leq p<\infty$ and consider $\{T_t\}$ acting on $H^p(\U)$. Then:  \\
(i) $\sigma_{\pi}(\Gamma)$ is empty.\\
(ii) $\sigma(\Gamma)=i\mathbb{R}$.\\
In particular $\Gamma$ is an unbounded operator.
\end{proposition}

\begin{proof}
(i) Let $\gamma$ be  an eigenvalue of $\Gamma$ and let $f$ be a corresponding eigenvector.
The eigenvalue equation $\Gamma(f)=\gamma  f$ is equivalent to the differential equation
$$
zf'(z)+(\gamma+\frac{1}{p})f(z)=0.
$$
The nonzero analytic solutions of this equation on $\U$ have the
form $f(z)=cz^{-(\gamma+\frac{1}{p})}$ with $c\ne 0$, which by Lemma \ref{(z)^l} are not in $H^p(\U)$.
It follows that $\sigma_{\pi}(\Gamma)=\emptyset$.

(ii) Because each $T_t$ is an invertible isometry its spectrum satisfies
$$
\sigma(T_t)\subseteq \partial\mathbb{D}.
$$
The spectral mapping theorem for strongly continuous groups (\cite[Theorem 2.3]{Pa}) says
$$
e^{t\sigma(\Gamma)}\subseteq\sigma(T_t).
$$
Thus if  $w \in \sigma(\Gamma)$, then  $e^{tw}\in \partial\mathbb{D}$, so that
$\sigma(\Gamma)\subseteq i\mathbb{R}$.
We will show that in fact $\sigma(\Gamma)=i\mathbb{R}$.

Let $\mu\in i\mathbb{R}$ and assume that $\mu\in \rho(\Gamma)$. Let  $\lambda=\mu +\frac{1}{p}$ and consider the function
$$
f_{\lambda}(z)=i\lambda(i+z)^{-\lambda-1}.
$$
Since $\textrm{Re}(-\lambda-1)=-1-1/p<-1/p$, this function is in $H^p(\U)$.
Since  $\mu\in \rho(\Gamma)$ the operator $R_{\mu}=(\mu-\Gamma)^{-1}: H^p(\U)\to H^p(\U)$ is  bounded.
 The image function  $g=R_{\mu}(f_{\lambda})$ satisfies the equation
$(\mu-\Gamma)(g)= f_{\lambda}$ or equivalently,
$$
(\mu+\frac{1}{p})g(z) +zg'(z)=f_{\lambda}(z), \quad z\in \U.
$$
Thus $g$ solves the differential equation
$$
\lambda g(z)+zg'(z)=i\lambda(i+z)^{-\lambda-1}, \quad z\in \U.
$$
It is easy to check that the general solution of this equation in $\U$ is
$$
G(z)= cz^{-\lambda} + (i+z)^{-\lambda}, \quad c\,\,\,\mbox{a constant}.
$$
Using the notation  of Lemma \ref{(z)^l} we find
$$
(\psi'(z))^{1/p}G(\psi(z))= \frac{c_1+cc_2(1+z)^{-\lambda}}{(1-z)^{\frac{2}{p}-\lambda}},
$$
where $c_1, c_2$ are nonzero constants. This last expression represents an analytic function on the unit disc which however,
by an argument similar to that in the proof of Lemma \ref{(z)^l},  is not  in the Hardy space  $H^p(\D)$
of the unit disc for any value of $c$ because
$\textrm{Re}(\lambda)=1/p$. Thus $G(z)$ is not in $H^p(\U)$ for any $c$ and this is a contradiction.
It follows that $\sigma(\Gamma)=i\mathbb{R}$, and this completes the proof.
\end{proof}

\section{The Ces\`{a}ro operators}\label{The Cesaro operator}

It follows from the above that when  $1<p<\infty$ the point
$\lambda_p=1-1/p$ is in the resolvent set of the generator
$\Gamma$. The resolvent operator $R(\lambda_p, \Gamma)$  is
therefore bounded. Let $f\in H^p(\U)$ and let $g=R(\lambda_p,
\Gamma)(f)$. It follows that $(\lambda_p-\Gamma)(g)=f$ or
equivalently
$$
(1-1/p)g(z)+zg'(z)+(1/p)g(z)=f(z).
$$
Thus $g$ satisfies the differential equation
$$
g(z)+zg'(z)=f(z), \quad z\in \U.
$$
Fix a point $w$ on the imaginary axis. Then
$$
zg(z)= \int_{w}^zf(\zeta)\,d\zeta +c, \quad z\in \U,
$$
with $c$ a constant. Now let $z=iy\to 0$  along the imaginary axis. Since  $g\in H^p(\U)$ with  $p>1$,  the estimate  (\ref{growth estimate of H^p})
implies that  $zg(z)\to 0$ therefore
$$
c=-\int_w^0f(\zeta)\,d\zeta
$$
(the existence of this integral is also a consequence of (\ref{growth estimate of H^p})). It follows that the integral of $f$ on the segment $[0, z]$ exists and we have
$$
g(z)=R(\lambda_p, \Gamma)(f)(z)
=\frac{1}{z}\int_0^zf(\zeta)\,d\zeta.
$$

\begin{theorem}\label{Cesaro} Let $1<p<\infty$ and let $\mathcal{C}$  be the operator defined by
$$
\mathcal{C}(f)(z)=\frac{1}{z}\int_{0}^{z}f(\zeta)\,d\zeta, \quad
f\in H^{p}(\U).
$$
Then $\mathcal{C}: H^p(\U)\to H^p(\U)$ is bounded. Further\newline
$$
\n{\mathcal{C}}=\frac{p}{p-1},
$$
and
$$
\sigma(\mathcal{C})=\{w\in\C:\abs{w-\frac{p}{2(p-1)}}=\frac{p}{2(p-1)}\}.
$$
\end{theorem}

\begin{proof}
As found  above, for $1<p<\infty$,
$$
\mathcal{C}=R(\lambda_p, \Gamma), \qquad \lambda_p=1-1/p\in \rho(\Gamma).
$$
The spectral mapping theorem \cite[Lemma VII.9.2]{DS}
gives
\begin{align*}
\sigma(\mathcal{C})&=\{\frac{1}{\lambda_p-z}: z\in \sigma(\Gamma)\}\cup \{0\}\\
&= \{\frac{1}{1-1/p-ir}: r\in \mathbb{R}\}\cup \{0\}\\
&= \{w\in\C: |w-\frac{p}{2(p-1)}|=\frac{p}{2(p-1)}\},
\end{align*}
giving the spectrum of $\mathcal{C}$ on $H^p(\U)$.

Since the spectral radius of $\mathcal{C}$ is equal to $p/(p-1)$ it follows that
$$
\n{\mathcal{C}}\geq  \frac{p}{p-1}.
$$
On the other hand we can apply the Hille-Yosida-Phillips theorem
to the group $\{T_t\}$ of isometries \cite[Corollary
VIII.1.14]{DS} to obtain
$$
\n{\mathcal{C}}=\n{R(\lambda_p, \Gamma)}\leq
\frac{1}{1-\frac{1}{p}}=\frac{p}{p-1},
$$
and the proof is complete.
\end{proof}

\textbf{Remark.} It follows from Lemma (\ref{(z)^l}) that   $h(z)=(i+z)^{-2}$ belongs to $ H^1(\U)$. The transformed function
$-i(i+z)^{-1}=\frac{1}{z}\int_{0}^{z}h(\zeta)\,d\zeta$ is  analytic
on $\U$ but  does not belong to $H^1(\U)$. Thus $\mathcal{C}$ does
not take $H^{1}(\U)$ to $H^{1}(\U)$.

\vspace{0,5cm}

We now consider the negative part $\{T_t: t\leq 0\}$ of the group $\{T_t\}$ and we rename it $\{S_t\}$. That is, for $f\in H^p(\U)$,
$$
S_t(f)(z) =e^{\frac{t}{p}}f(\phi_{-t}(z))=e^{\frac{t}{p}}f(e^tz), \quad t\geq 0.
$$
It is clear that $\{S_t\}$ is strongly continuous on $H^p(\U)$ and that its generator is $\Delta=-\Gamma$. It follows from  Proposition \ref{spectr(Ã)} that
$$
\sigma(\Delta)=i\R, \qquad \sigma_{\pi}(\Delta)=\emptyset.
$$
Let $\mu_p=\frac{1}{p}\in \rho(\Delta)$, and consider the bounded resolvent operator $R(\mu_p,  \Delta)$. Let $f\in H^p(\U)$ and let
$g=R(\mu_p,  \Delta)(f)$. Then
$\mu_p g-\Delta(g)=f$ or equivalently
$$
zg'(z)=-f(z), \quad z\in \U.
$$
Fix a point $a\in \U$, then we have
$$
g(z)= -\int_a^z\frac{f(\zeta)}{\zeta}\,d\zeta + c
$$
with $c$ a constant. Now let $z\to \infty$  within a half-plane
$\textrm{Im}(z)\geq \delta >0$. Then $g(z)\to 0$ \cite[Corollary
1, page 191]{Du}, thus the following limit exists and gives the
value of $c$,
$$
c=\lim_{\underset{\textrm{Im}(w)\geq
\delta}{w\to\infty}}\,\int_a^{w}\frac{f(\zeta)}{\zeta}\,d\zeta.
$$
We therefore find
\begin{align*}\label{lim}
g(z)=&\int_z^a\frac{f(\zeta)}{\zeta}\,d\zeta+\lim_{\underset{\textrm{Im}(w)\geq \delta}{w\to\infty}}\,\int_a^w\frac{f(\zeta)}{\zeta}\,d\zeta\\
=&\lim_{\underset{\textrm{Im}(w)\geq
\delta}{w\to\infty}}\,\int_z^w\frac{f(\zeta)}{\zeta}\,d\zeta.
\end{align*}
We now show that the above limit exists unrestrictedly when
$w\to\infty$ within the half-plane $\U$. Indeed for $\delta>0$ we
have,
$$
\int_z^{w}\frac{f(\zeta)}{\zeta}\,d\zeta=\int_z^{w+i\delta}\frac{f(\zeta)}{\zeta}\,d\zeta
+\int_{w+i\delta}^{w}\frac{f(\zeta)}{\zeta}\,d\zeta = J_w+I_w.
$$
With the change of variable $\zeta=w+is\delta$, $0\leq s\leq 1$,
we have
\begin{align*}
|I_w|&=|-i\delta\int_0^1\frac{f(w+is\delta)}{w+is\delta}\,ds|\\
&\leq  \delta\int_0^1\frac{|f(w+is\delta)|}{|w+is\delta|}\,ds\\
&\leq  \frac{\delta}{|w|}\int_0^1|f(w+is\delta)|\,ds\\
&= \frac{1}{|w|}\int_{y}^{y+\delta}|f(x+iu)|\,du,
\end{align*}
where $w=x+iy$ and $u=y+s\delta$. For $p=1$ the Fej\'{e}r-Riesz
inequality for the half-plane \cite[Exercise 6, page 198]{Du}
implies, through the last inequality,
$$
|I_w|\leq \frac{1}{|w|}\int_0^{\infty}|f(x+iu)|\,du\leq
\frac{1}{2|w|}\n{f}_1 \to 0, \quad \mbox{as}\,\,w\to \infty,
$$
while for $p>1$, the growth estimate (\ref{growth estimate of
H^p}) gives
\begin{align*}
|I_w|&\leq C_p\left(\frac{1}{|w|}\int_{y}^{y+\delta}u^{-\frac{1}{p}}\,du \right) \n{f}_p\\
&\leq C_p'\frac{(y+\delta)^{1/q}}{|w|}\n{f}_p, \quad (1/q=-1/p+1),\\
&\leq C_p'\frac{(|w|+\delta)^{1/q}}{|w|}\n{f}_p \to 0, \quad
\mbox{as}\,\,w\to \infty.
\end{align*}
Thus in all cases $I_w\to 0$ and we have
$$
\lim_{w\to\infty}\int_z^{w}\frac{f(\zeta)}{\zeta}\,d\zeta=
\lim_{w\to\infty}J_w +\lim_{w\to\infty}I_w =g(z)
$$
i.e. the unrestricted limit exists. For this reason we can write
$$
g(z)= \int_z^{\infty}\frac{f(\zeta)}{\zeta}\,d\zeta.
$$

\begin{theorem}\label{results on T}
Let $1\leq p<\infty$ and  let $\mathcal{T}$  be the operator defined by
$$
\mathcal{T}(f)(z)=\int_{z}^{\infty}\frac{f(\zeta)}{\zeta}\,d\zeta,
\quad f\in H^{p}(\U).
$$
Then $\mathcal{T}: H^p(\U)\to H^p(\U)$ is bounded. Further\newline
$$
\n{\mathcal{T}}=p,
$$
and
$$
\sigma(\mathcal{T})=\{w\in\C:\abs{w-\frac{p}{2}}=\frac{p}{2}\}.
$$
\end{theorem}

\begin{proof} As found above we have
$$
\mathcal{T}=R(\mu_p, \Delta), \quad \mu_p=1/p\in \rho(\Delta).
$$
From the  spectral mapping theorem we find
\begin{align*}
\sigma(\mathcal{T})&=\{\frac{1}{\mu_p-z}: z\in \sigma(\Delta)\}\cup\{0\}\\
&= \{\frac{1}{1/p-ir}: r\in \mathbb{R}\}\cup\{0\}\\
&= \{w\in\C: |w-\frac{p}{2}|=\frac{p}{2}\},
\end{align*}
giving the spectrum of $\mathcal{T}$ on $H^p(\U)$.

Since the spectral radius of $\mathcal{T}$ is equal to $p$ it follows that
$$
\n{\mathcal{T}}\geq  p.
$$
On the other hand we can apply the Hille-Yosida-Phillips theorem
to the semigroup $\{S_t\}$ of isometries  to obtain
$$
\n{\mathcal{T}}=\n{R(\mu_p, \Delta)}\leq \frac{1}{1/p}=p,
$$
and the proof is complete.
\end{proof}

Suppose now $1<p<\infty$ and let $q$ be the conjugate index, $1/p+1/q=1$.   Recall  the duality $(H^p(\U))^{*}=H^q(\U)$ which
is realized through the pairing
\begin{equation}\label{pairing}
\langle f, g\rangle
=\int_{-\infty}^{\infty}f^*(x)\overline{g^*(x)}\,dx.
\end{equation}
 For the semigroups $T_t(f)(z)= e^{-\frac{t}{p}}f(e^{-t}z)$ and $S_t(g)(z)=e^{\frac{t}{q}}g(e^{t}z)$ acting on $H^p(\U)$ and $H^q(\U)$ respectively we find
\begin{align}\label{adjoint semigroups}
\langle T_t(f), g\rangle
&=\int_{-\infty}^{\infty}e^{-\frac{t}{p}}f^*(e^{-t}x)\overline{g^*(x)}\,dx\nonumber\\
&=\int_{-\infty}^{\infty}f^*(x)\overline{e^{\frac{t}{q}}g^*(e^tx)}\,dx\\
&=\langle f, S_{t}(g)\rangle \nonumber.
\end{align}
Thus $\{T_t\}$ and $\{S_t\}$ are adjoints of each other. From the general theory of operator semigroups this relation of being adjoint, on reflexive spaces, is inherited by the infinitesimal generators and subsequently by the resolvent operators \cite[Corollaries  10.2 and 10.6]{Pa}. It follows that
$\mathcal{C}$ and $\mathcal{T}$ are adjoints of each other on the reflexive Hardy spaces of the half-plane.

\section{Boundary correspondence}\label{Boundary values - unitary equivalence}

 We now  examine the boundary correspondence between $\mathcal{C}$ and its real line version  $\mathtt{C}$, as well as between $\mathcal{T}$ and the corresponding real line operator $\mathtt{T}$ defined on $L^p(\R)$ by
\begin{equation*}
\mathtt{T}(f)(x)=\left\{
\begin{array}{ll}
    \int_{x}^{\infty}\frac{f(u)}{u}\,du, & \hbox{$x>0$,}
    \\ \\
   -\int_{-\infty}^x\frac{f(u)}{u}\,du, & \hbox{$x<0,$} \\
\end{array}
\right.
\end{equation*}
while $\mathtt{T}(f)(0)$ can be chosen arbitrarily. It is well known (and easy to prove) that $\mathtt{T}$ is bounded on $L^p(\R)$
for $p\geq 1$.

\begin{theorem}\label{Boundary} Consider the operators $\mathcal{C}$, $\mathcal{T}$ on the Hardy spaces $H^p(\U)$ and the operators
$\mathtt{C}$, $\mathtt{T}$ on the spaces $L^p(\R)$. Then\\
(i) For $1<p<\infty$ and $f\in H^{p}(\U)$,
$$
\mathcal{C}(f)^{\ast}=
\mathtt{C}(f^{\ast}).
$$
(ii) For $1\leq p<\infty$ and $f\in H^{p}(\U)$,
$$
\mathcal{T}(f)^{\ast}=\mathtt{T}(f^{\ast}).
$$
\end{theorem}

\begin{proof}
(i) For $f\in H^p(\U)$ and $\delta>0$ consider the  function
$$
f_{\delta}(z)=f(z+i\delta), \quad \textrm{Im}(z)>-\delta.
$$
It is clear that $f_{\delta} \in H^{p}(\U)$ and  $f_{\delta}^{\ast}(x)=f_{\delta}(x)$,  and we have
\begin{align*}
\n{\mathcal{C}(f)^{\ast}-\mathtt{C}(f^{\ast})}_{L^p(\R)}&\leq
\n{\mathcal{C}(f)^{\ast}-\mathcal{C}(f_{\delta})^{\ast}}_{L^p(\R)}+
\n{\mathcal{C}(f_{\delta})^{\ast}-
\mathtt{C}(f_{\delta}^{\ast})}_{L^p(\R)}\\
&\,\,\,\,\,\qquad \qquad \qquad\qquad\qquad+ \n{\mathtt{C}(f_{\delta}^{\ast})-\mathtt{C}(f^{\ast})}_{L^p(\R)}\\
&\leq
(\n{\mathcal{C}}+\n{\mathtt{C}})\n{f_{\delta}-f}_{H^p(\U)}+\n{\mathcal{C}(f_{\delta})^{\ast}-
\mathtt{C}(f_{\delta}^{\ast})}_{L^p(\R)}.
\end{align*}
We can make  $\n{f_{\delta}-f}_{H^p(\U)}$ as small as we wish by choosing $\delta$ close enough to $0$. Thus in order to prove that
$\mathcal{C}(f)^*(x)=\mathtt{C}(f^*)(x)$ a.e. on $\R$, it
suffices to show
$\mathcal{C}(f_{\delta})^*(x)=\mathtt{C}(f_{\delta}^*)(x)$, i.e.
$$
\mathcal{C}(f_{\delta})^*(x)=\frac{1}{x}\int_0^x f_{\delta}^{\ast}(u)\,du
$$
for almost all $x$. Now for $z=x+iy\in\U$, since  $f_{\delta}(z)$ is analytic on $\{\mbox{Im}(z)>-\delta\}$, its integral on the segment $[0, z]$ can be obtained by integrating over the path $[0,x]\cup[x,z]$, so we have
\begin{align*}
\mathcal{C}(f_{\delta})(z)&= \frac{1}{z}\int_0^z
f_{\delta}(\zeta)\,d\zeta\\
&= \frac{1}{z}\int_{[0,\,x]}f_{\delta}(\zeta)\,d\zeta+
\frac{1}{z}\int_{[x,\,z]}f_{\delta}(\zeta)\,d\zeta.
\end{align*}
If $x\ne 0$ then clearly the limit of the first integral as $y\to 0$ is
$$
\lim_{y\to 0}\frac{1}{z}\int_{[0,\,x]}f_{\delta}(\zeta)\,d\zeta=\frac{1}{x}\int_0^xf_{\delta}(u)\,du.
$$
The  limit of the second integral  vanishes. Indeed  since  $f\in H^p(\U)$, $f$  is bounded over every  half-plane $\{z: \mbox{Im}(z)\geq \delta\}$  and we find
$$
\sup_{\zeta\in [x,\,z]}|f_{\delta}(\zeta)|\leq
\sup_{\textrm{Im}(z)\geq \delta}|f(z)|=M<\infty,
$$
therefore
$$|\frac{1}{z}\int_{[x,\,z]}f_{\delta}(\zeta)\,d\zeta|\leq \frac{1}{|x|}(\sup_{\zeta\in [x,\,z]}|f_{\delta}(\zeta)|)y\leq \frac{M}{|x|}y \to 0$$ as $y\to 0$. It follows that
$$
 \mathcal{C}(f_{\delta})^{\ast}(x)=\lim_{y\to 0}\mathcal{C}(f_{\delta})(z)= \frac{1}{x}\int_0^xf_{\delta}(u)\,du=\frac{1}{x}\int_0^xf_{\delta}^{\ast}(u)\,du
$$
and the proof of (i) is complete.

(ii) We argue as in part (i) and use the same notation. If $f\in H^p(\U), 1\leq p<\infty$, let $f_{\delta}$ be defined as in part (i). Using the triangle inequality as in part (i) we see that it suffices to show that for almost all $x$,
$$
\mathcal{T}(f_{\delta})^{\ast}(x)=\mathtt{T}(f_{\delta}^{\ast})(x).
$$
Let  $z=x+iy\in \U$ with $x>0$ (the case $x<0$ is similar),  choose $s>0$, and consider the path $[z,x]\cup[x, x+s]\cup[x+s, z+s(1+i)]$ as an alternative path of integration to obtain
\begin{align*}
\mathcal{T}(f_{\delta})(z)&=\lim_{s\to\infty}\int_z^{z+s(1+i)}
\frac{f_{\delta}(\zeta)}{\zeta}\,d\zeta\\
&=\int_z^{x}\frac{f_{\delta}(\zeta)}{\zeta}\,d\zeta+
\lim_{s\to\infty}\left(\int_{x}^{x+s}\frac{f_{\delta}(\zeta)}{\zeta}\,d\zeta+
\int_{x+s}^{z+s(1+i)}\frac{f_{\delta}(\zeta)}{\zeta}\,d\zeta\right).
\end{align*}
For the first integral inside the limit it is clear that
$$
\lim_{s\to\infty}\int_{x}^{x+s}\frac{f_{\delta}(\zeta)}{\zeta}\,d\zeta= \int_{x}^{+\infty}\frac{f_{\delta}(u)}{u}\,du.
$$
Write  $I(s)=\int_{x+s}^{z+s(1+i)}\frac{f_{\delta}(\zeta)}{\zeta}\,d\zeta$, the second integral inside the limit, then
\begin{align*}
|I(s)|
&=|\int_0^{y+s}\frac{f_{\delta}(x+s+it)}{x+s+it}i\,dt|\\
&\leq \int_0^{y+s}\frac{|f_{\delta}(x+s+it)|}{|x+s+it|}\,dt\\
&\leq \frac{1}{x+s}\int_0^{y+s}|f_{\delta}(x+s+it)|\,dt.
\end{align*}
If $p=1$ then the Fej\'{e}r -Riesz inequality for the upper half-plane gives
$$
\int_0^{y+s}|f_{\delta}(x+s+it)|\,dt\leq \int_0^{\infty}|f_{\delta}(x+s+it)|\,dt
\leq  \frac{1}{2}\n{f_{\delta}}_1
\leq \frac{1}{2}\n{f}_1,
$$
thus $I(s)\leq \frac{1}{2(x+s)}\n{f}_1 \to 0$ as $s\to \infty$. If $p>1$ then using  \eqref{growth estimate of H^p} we obtain
\begin{align*}
\int_0^{y+s}|f_{\delta}(x+s+it)|\,dt & \leq C_p\n{f_{\delta}}_p\int_0^{y+s}\frac{1}{t^{1/p}}\,dt\\
&\leq C_p'\n{f}_p(y+s)^{-\frac{1}{p}+1},
\end{align*}
which implies $I(s)\leq C_p'\n{f}_p\frac{(y+s)^{-\frac{1}{p}+1}}{x+s}\to 0$ as $s\to \infty$. We have shown
$$
\mathcal{T}(f_{\delta})(z)=
\int_z^{x}\frac{f_{\delta}(\zeta)}{\zeta}\,d\zeta+
\int_{x}^{+\infty}\frac{f_{\delta}(u)}{u}\,du
$$
for each $z=x+iy\in\U$. Taking the limit of the above as $y\to 0$ we find
$$
\mathcal{T}(f_{\delta})^{\ast}(x)=
\int_{x}^{+\infty}\frac{f_{\delta}(u)}{u}\,du
=\mathtt{T}(f_{\delta}^{\ast})(x)
$$
and the proof is complete.
\end{proof}

As a consequence we have the following

\begin{corollary} For $1\leq p<\infty$ let  $\mathcal{H}_p$ be the closed subspace of $L^p(\R)$ consisting of all  boundary functions $f^{\ast}$ of
$f\in H^p(\U)$. Then
$\mathtt{C}(\mathcal{H}_p)\subset \mathcal{H}_p$ for  $1<p<\infty$, and
$\mathtt{T}(\mathcal{H}_p)\subset \mathcal{H}_p$ for $1\leq p<\infty$.
\end{corollary}

\bibliographystyle{amsalpha}

\end{document}